\documentclass[11pt]{article}
\usepackage{amsmath,amsthm,amssymb,amsfonts}
\usepackage{latexsym}
\usepackage{graphicx,psfrag,import}
\usepackage{fullpage}
\usepackage{framed}
\usepackage{verbatim}
\usepackage{color}
\usepackage{epsfig}
\usepackage{epstopdf}
\usepackage{hyperref}
\usepackage{geometry}
\usepackage{mathtools}
\usepackage{nonfloat}
\usepackage{enumerate}
\usepackage{multicol}
\usepackage{a4wide}
\usepackage{booktabs}
\usepackage{enumitem}
\usepackage{lineno}
\usepackage{parcolumns}
\usepackage{thmtools}
\usepackage{xr}
\usepackage{epstopdf}
\usepackage{mathrsfs}
\usepackage{subfig}
\usepackage{caption}
\usepackage{comment}
\usepackage{authblk}
\usepackage{setspace}
\usepackage{algorithm}
\usepackage{algpseudocode}

\theoremstyle{plain}
\newtheorem{theorem}{Theorem}[section]
\newtheorem{corollary}[theorem]{Corollary}
\newtheorem{proposition}[theorem]{Proposition}
\newtheorem{lemma}[theorem]{Lemma}

\theoremstyle{definition}
\newtheorem{definition}[theorem]{Definition}

\newtheorem{remark}[theorem]{Remark}

\theoremstyle{remark}

\newcommand\GG{\mathcal{G}}

\newcommand\by{\boldsymbol{y}}
\newcommand\bw{\boldsymbol{w}}

\newcommand\bu{\boldsymbol{u}}
\newcommand\bx{\boldsymbol{x}}
\newcommand\bv{\boldsymbol{v}}

\newcommand\bk{\boldsymbol{k}}

\begin{document}

\title{Source-only realizations, weakly reversible deficiency one networks and dynamical equivalence}

\author{Abhishek Deshpande}
\affil[1]{Center for Computational Natural Sciences and Bioinformatics, International Institute of Information Technology Hyderabad, {\tt abhishek.deshpande@iiit.ac.in}}

\maketitle

\begin{abstract}

Reaction networks can display a wide array of dynamics. However, it is possible for different reaction networks to display the same dynamics. This phenomenon is called \emph{dynamical equivalence} and makes network identification a hard problem to solve. We show that to find a \emph{strongly endotactic, endotactic, consistent, or conservative} realization that is dynamically equivalent to the mass-action system generated by a given network, it suffices to consider only the source vertices of the given network. In addition, we show that weakly reversible deficiency one realizations are not unique. We also present a  characterization of the dynamical relationships that exist between several types of weakly reversible deficiency one networks.

\end{abstract}

\section{Introduction}

It is in general difficult to analyze the dynamics exhibited by nonlinear systems. They can exhibit a wide range of behaviour ranging from periodicity, chaos, limit cycles etc. One way of analyzing dynamical systems is to represent them as reaction networks. Reaction networks are ubiquitous in biological systems. The dynamics generated by reaction networks often assume polynomial or power-law form. A key feature here is how the graphical structure of the reaction network regulates its dynamics. Such graphical properties include but are not limited to weak reversibility, complex balancing, endotacticity etc. Reaction networks possessing these properties are conjectured to to exhibit dynamical properties like persistence (which implies that no species can go extinct), permanence (which implies convergence to a compact set) and the existence of a globally attracting steady state.

It turns out that there exists networks with different graphical structures that exhibit the same dynamics. This property is called \emph{dynamical equivalence} and makes the problem of network identification difficult. Dynamical equivalence has also been called 
``macro-equivalence'' by Horn and Jackson~\cite{horn1972general}  and ``confoundability'' by Craciun and Pantea~\cite{craciun2008identifiability}. A natural application is to infer the dynamics of networks that do not possess properties like weak reversibility or endotacticity, but are dynamically equivalent to weakly reversible or endotactic networks. A large body of literature exists on designing efficient algorithms for detecting weakly reversible or endotactic realizations, given a dynamical system and a set of source vertices. One of our contributions in this paper is to give an upper bound on the number of source vertices required for detecting certain families of reaction networks.

Another quantity that is useful in the analysis of reaction networks is the deficiency of the reaction network. Informally, deficiency is a measure of how far the reaction vectors in the network are from their general position. In particular, weakly reversible deficiency zero reaction networks exhibit robust dynamics. It is known that for weakly reversible deficiency zero reaction networks, there exists a Lyapunov function. In addition, there exists a unique steady state that is locally asymptotically stable~\cite{horn1972general}. Recently, it was shown~\cite{craciun2021uniqueness} that weakly reversible deficiency zero realizations are unique, i.e., there cannot exist two different weakly reversible deficiency zero networks that generate the same dynamical system. Further, an algorithm has been proposed to check whether a weakly reversible deficiency zero realization exists given a dynamical system and a set of source vertices~\cite{craciunalgorithm}. A natural extension is to ask whether weakly reversible deficiency one realizations are unique. We show that the answer to this question is no. Another question worth exploring is to analyze the property of dynamical equivalence between different types of weakly reversible deficiency one networks, and between weakly reversible deficiency one and weakly reversible deficiency zero networks. We attempt to answer these questions in this paper.

Our paper is organized as follows: In Section~\ref{sec:euclidean_graphs} we give a short primer on reaction networks. More specifically we introduce reaction networks as graphs embedded in Euclidean space. We define weakly reversible, strongly endotactic, endotactic, consistent and conservative networks. In Section~\ref{sec:net_vectors} we introduce \emph{net vectors} and \emph{ghost vertices}. The idea of network corresponds to the net flux out a vertex. A ghost vertex is a source vertex whose net vector is zero. We will use both these notions later in the paper to prove results about source-only realizations and dynamical equivalence between different  families of reaction networks. In Section~\ref{sec:dynamical_equivalence} we define the notion of \emph{dynamical equivalence}. We show in Lemma~\ref{lem:weak_stoichiometric} that two dynamically equivalent weakly reversible or strongly endotactic realizations have the stoichiometric subspace. In Section~\ref{sec:source_only} in Theorem~\ref{thm:main_source_only} we show that certain families of reaction networks possess the following property: A mass-action system $\GG_k$ is dynamically equivalent to a mass-action system generated by networks from this family iff it is dynamically equivalent to a mass-action system $\tilde{\GG}_{\tilde k}$ generated by networks from this family that uses only the source vertices of $\GG$. The family of such networks include weakly reversible, consistent, endotactic and strongly endotactic networks. Section~\ref{sec:uniqueness} deals with the uniqueness and non-uniqueness of realizations. In particular, we show that weakly reversible deficiency one realizations are not unique, i.e., there exists different weakly reversible deficiency one networks that generate the same dynamics. In Theorem~\ref{thm:zero_dyn_one} we show that two weakly reversible deficiency one realizations with the property that one realization has all linkage classes of deficiency zero and the other realization has all linkage classes of deficiency zero except one which has deficiency one, cannot be dynamically equivalent.

\section{Reaction Networks}\label{sec:euclidean_graphs}

Reaction networks can be characterized as directed graphs in Euclidean space called \textit{Euclidean embedded graphs}~\cite{craciun2015toric,craciun2019polynomial,craciun2020endotactic,craciun2019quasi}. In particular, a reaction network is a tuple $\GG = (V, E)$, where $V\subset\mathbb{R}^n$ is the set of vertices (that correspond to complexes) and $E$ is the set of edges (that correspond to reactions). An edge $(\by,\by')\in E$ corresponds to a reaction $\by\to \by'$. Here $\by$ is the source vertex and $\by'$ is the target vertex. 

\begin{figure}[h!]
\centering
\includegraphics[scale=0.45]{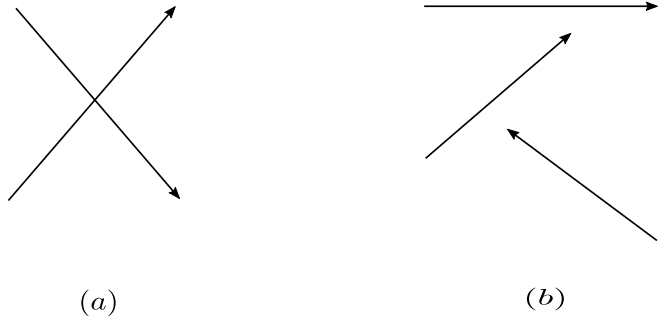}
\caption{\small This figures shows a few examples of Euclidean embedded graphs.}
\label{fig:E_graph}
\end{figure} 

The stoichiometric subspace corresponding to a reaction network $\GG$ is given by the set $S: = \rm{span}\{\by'-\by\, |\, \by\to \by'\in E \}$. Given $\bx_0\in\mathbb{R}^n_{>0}$, the stoichiometric compatibility class of $\bx_0$ is the polyhedron $(\bx_0 + S)\cap\mathbb{R}^n_{>0}$.

A \emph{linkage} class $\mathcal{L}$ is a maximal connected component of $\GG$. The \emph{deficiency} of a reaction is the non-negative integer $\delta$, given by the formula $\delta = |V| - \ell - \rm{dim}(S)$, where $|V|$ is the number of vertices, $\ell$ is the number of linkage classes and $\rm{dim}(S)$ is the dimension of the stoichiometric subspace. Let $V_\mathcal{S}$ denote the set of source vertices and let $\rm{conv}(V_\mathcal{S})$ denote the convex hull of the source vertices of $\GG$. Then

\begin{itemize}

\item $\GG$ is \textit{weakly reversible} if every reaction is part of a strongly connected component.

\item $\GG$ is \textit{strongly endotactic}~\cite{gopalkrishnan2014geometric} if for every vector $\bv\in\mathbb{R}^n$ and $\by\to \by'\in E$ with $\bv\cdot (\by' -\by) < 0$, there exists $\tilde{\by}\to \tilde{\by}'\in E$ such that $\bv\cdot\tilde{\by} < \bv\cdot \by$, $\bv\cdot(\tilde{\by}\to \tilde{\by}') >0$ and $\bv\cdot\tilde{\by} < \bv\cdot \hat{\by}$ for all $\hat{\by}\in V_\mathcal{S}$.

\item  $\GG$ is \textit{endotactic}~\cite{craciun2013persistence} if for every vector $\bv\in\mathbb{R}^n$ and $\by\to \by'\in E$ with $\bv\cdot (\by' -\by) < 0$, there exists $\tilde{\by}\to \tilde{\by}'\in E$ such that $\bv\cdot\tilde{\by} < \bv\cdot \by$, $\bv\cdot(\tilde{\by}\to \tilde{\by}') >0$. 

\item  $\GG$ is \textit{consistent} if there exists positive constants $\lambda_{\by\to \by'}$ such that $\displaystyle\sum_{\by\to \by'\in E} \lambda_{\by\to \by'}(\by' - \by) = 0$. 

\item $\GG$ is \textit{conservative} if there exists a vector $\bv\in \mathbb{R}^n_{>0}$ such that $\bv\in S^{\perp}$. 

\end{itemize}

Figure~\ref{fig:reaction_types} shows a few examples of such networks.

\begin{figure}[h!]
\centering
\includegraphics[scale=0.45]{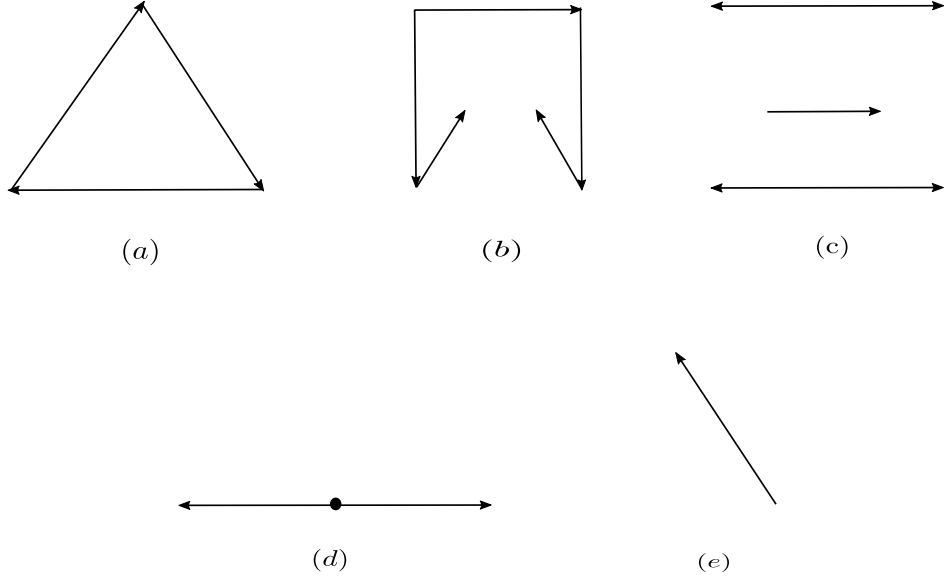}
\caption{\small (a) Weakly reversible reaction network. Every reaction in this network is part of a strongly connected component. (b) Strongly endotactic reaction network. (c) Endotactic reaction network. (d) Consistent reaction network. The sum of the reaction vectors in this network is zero. (e) Conservative reaction network. There exists a vector $(1,1)\in\mathbb{R}^2_{>0}$ that is orthogonal to the stoichiometric subspace of the network (which is the one-dimensional subspace spanned by the vector $(-1,1)$).}
\label{fig:reaction_types}
\end{figure}

The chain of inclusions among the above mentioned networks is as follows: Weakly reversible networks consisting of a single linkage class are strongly endotactic. Weakly reversible networks networks are endotactic. Strongly endotactic networks are also endotactic. Further endotactic networks are consistent. It is possible for a network to be 

\renewcommand{\theenumi}{(\roman{enumi})}
\begin{enumerate}

\item consistent, but not conservative

\item conservative, but not consistent.

\item consistent and conservative

\item neither consistent nor conservative.

\end{enumerate}

Figure~\ref{fig:consistent_conservative} gives examples of networks that are of the four types mentioned above.

\begin{figure}[h!]
\centering
\includegraphics[scale=0.4]{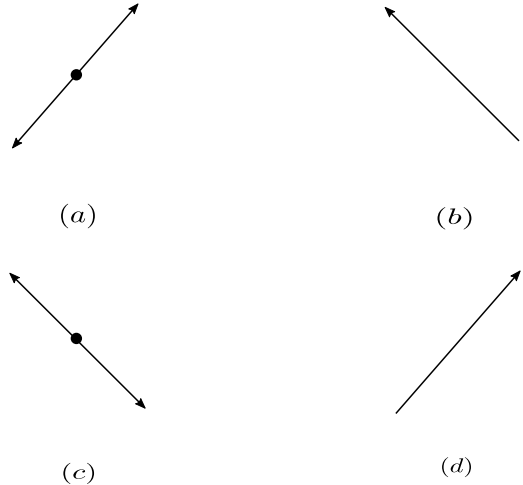}
\caption{\small (a) Consistent but not conservative network. (b) Conservative but not consistent network. (c) Consistent and conservative network (d) Neither consistent nor conservative network.}
\label{fig:consistent_conservative}
\end{figure}

Conservative networks are also related to the notion of \emph{critical sets}, which are relevant in the context of autocatalysis~\cite{deshpande2014autocatalysis,hordijk2010autocatalytic,hordijk2004detecting,hordijk2012structure}. It can be shown that if a network is conservative, then the set of species is not a critical set~\cite{deshpande2014autocatalysis}. Critical sets also show up in the context of persistence; it is known that if there are no \emph{critical siphons} in the network, then the dynamics exhibited by it is persistent~\cite{angeli2007petri}.

It turns out that there is a geometrically intuitive way to deduce whether the network is endotactic or strongly endotactic. This is given by the \emph{parallel sweep} test~\cite{craciun2013persistence,gopalkrishnan2014geometric}. The parallel sweep test for strongly endotactic networks can be summarized as follows~\cite{gopalkrishnan2014geometric}

\begin{proposition}

Let $\GG=(V,E)$ be a reaction network. For every $\bu\notin S^{\perp}$, we sweep with a hyperplane perpendicular to $\bu$ in the direction of $u$ until it hits the boundary of $\rm{conv}(V_\mathcal{S})$. Let $H_{\rm supp}$ denote this supporting hyperplane.
On $H_{\rm supp}$, if all reactions $\by\rightarrow \by'\in E$ with source vertex $\by\in H_{\rm supp}$ satisfy $\bw\cdot (\by'-\by)\geq 0$ and there exists a reaction $\tilde{\by}\rightarrow \tilde{\by}'\in E$ such that $\bw\cdot (\tilde{\by}' - \tilde{\by})>0$, then the network passes the parallel sweep test, i.e., the network is strongly endotactic. If the network fails to satisfy the parallel sweep test, it is not strongly endotactic.

\end{proposition}

 The parallel sweep test for endotactic networks can be summarized as follows~\cite{craciun2013persistence}

\begin{proposition}

Let $\GG=(V,E)$ be a reaction network. For every $\bu$, we sweep with a hyperplane $H$ perpendicular to $u$ in the direction of $u$ until it hits the boundary of $\rm{conv}(V_\mathcal{S})$. On $H$, check if all reactions $\by\rightarrow \by'\in E$ with source vertex $\by\in H$ satisfy $\bw\cdot (\by'-\by)\geq 0$. If yes, and there exists a reaction $\tilde{\by}\rightarrow \tilde{\by}'$ such that $\bw\cdot (\tilde{\by}' - \tilde{\by})>0$, then the network passes the parallel sweep test, i.e., the network is endotactic. If there exists a reaction $\by_0\rightarrow \by'_0\in E$ with $\by_0\in H$ such that $\bw\cdot (\by'_0 - \by_0)<0$, then the network fails to satisfy the parallel sweep test, i.e., it is not strongly endotactic. If we have $\bw\cdot (\by' - \by)=0$ for all reactions $\by\rightarrow \by'\in E$ with source vertex $\by\in H$, then continue sweeping with the hyperplane until you hit another source vertex and repeat the same steps as above. If you continue sweeping with the hyperplane and have covered all source vertices, then the network passes the parallel sweep test and the network is strongly endotactic.

\end{proposition}

A reaction network can generate a wide array of dynamical systems. If we assume mass-action kinetics~\cite{feinberg1979lectures,voit2015150,guldberg1864studies,yu2018mathematical,gunawardena2003chemical,adleman2014mathematics}, the dynamical system generated by $\GG$ is given by
\begin{eqnarray}\label{eq:mass_action}
\frac{d\bx}{dt} = \displaystyle\sum_{\by\rightarrow \by'\in E}k_{\by\rightarrow \by'}{\bx}^{\by}(\by' - \by)
\end{eqnarray}

Let us denote vector of rate constants in this dynamical system by $\bk$. Then the associated mass-action system will be represented as $(\GG,\bk)$. A mass-action system is \emph{complex balanced} if there exists a point $\bx_0\in\mathbb{R}^n_{>0}$ such that for any vertex $\by\in V$, we have 
\begin{eqnarray}
\displaystyle\sum_{\by\rightarrow\by'\in E} k_{\by\rightarrow\by'}{\bx_0}^{\by} = \displaystyle\sum_{\tilde{\by}\rightarrow\by\in E}k_{\tilde{\by}\rightarrow\by}{\bx_0}^{\tilde{\by}}
\end{eqnarray}

As remarked before, graphical properties like weak reversbility, endotacticity and complex balancing are closely related to dynamical properties like persistence, permanence and global stability. Anderson has proved the \emph{Global Attractor Conjecture}, which says that complex balanced systems have a globally attracting steady state within every stoichiometric compatibility class when the associated reaction graph consists of a single linkage class. It is know from the work of Craciun, Nazarov and Pantea~\cite{craciun2013persistence} that two-dimensional endotactic networks are permanent. This has been extended by Pantea~\cite{pantea2012persistence} to endotactic networks with two-dimensional subspace. Gopalkrishnan, Miller and Shiu~\cite{gopalkrishnan2014geometric} have shown that strongly endotactic networks are permament. Recently, Craciun~\cite{craciun2015toric} has proposed a proof of the \emph{Global Attractor Conjecture}.

\section{Net vectors and Ghost vertices}\label{sec:net_vectors}

The idea of this section is to introduce the notions of net vectors and ghost vertices, which will be used later in the paper to analyze the dynamics of several families of networks. The dynamical system generated by $(\GG,\bk)$ (given by Equation~\ref{eq:mass_action}) can be rewritten as
\begin{eqnarray}
\frac{d\bx}{dt} = \displaystyle\sum_{\by \in V_S} {\bx}^{\by}\displaystyle\sum_{\by'\in V}k_{\by\rightarrow \by'}{\bx}^{\by}(\by' - \by)
\end{eqnarray}
Then $\displaystyle\sum_{\by'\in V}k_{\by\rightarrow \by'}{\bx}^{\by}(\by' - \by)$ will denote the \textbf{net vector} corresponding to the source vertex $\by$.

\begin{remark}
Note that the net vector corresponding to the source vertex $\by$ may consist of terms like $(\by' - \by)$ where $\by\rightarrow\by'\notin E$. In such cases, we set $k_{\by\rightarrow \by'}=0$. 
\end{remark}

If the net reaction vector corresponding to a source vertex $\by_0$ is zero, then $\by_0$ will be called a \textbf{ghost} vertex. Previous works~\cite{craciun2020efficient} have referred to such vertices as \emph{virtual sources}. Figure~\ref{fig:ghost_vertex} shows an example of a ghost vertex.

\begin{figure}[h!]
\centering
\includegraphics[scale=0.45]{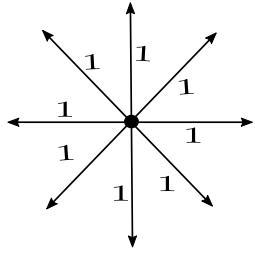}
\caption{\small This figure shows an example of ghost vertex; the net vector corresponding to this source vertex is zero.}
\label{fig:ghost_vertex}
\end{figure}

\begin{theorem}\label{thm:bound_ghost_vertex}
Consider a reaction network $\GG=(V,E)$ with deficiency $\delta$. Then $\GG$ can have at most $\delta$ ghost vertices.
\end{theorem}

\begin{proof}

For contradiction, assume that $\GG$ has $t$ ghost vertices where $t \geq \delta + 1$. By Remark~\ref{rmk:procedure}, applying the procedure in~\cite{craciun2021uniqueness}[Theorems 4.8 and 4.12] $\delta + 1$ times gives a reaction network of deficiency $-1$ and  $t - (\delta+1)$ ghost vertices, which is not possible. 

\end{proof}

\begin{corollary}\label{cor:def_zero_ghost}
Consider a reaction network $\GG=(V,E)$ with deficiency $0$. Then $\GG$ cannot have any ghost vertices.
\end{corollary}

\begin{proof}
Follows from Theorem~\ref{thm:bound_ghost_vertex}.
\end{proof}

\begin{lemma}\label{lem:strongly_endotactic_prop}
Consider a strongly endotactic mass-action system $(\GG,k)$. Let $\{\bw_i\}_{i=1}^m$ denote the net reaction vectors of $\GG$. Define a matrix $W$ with $\{\bw_i\}_{i=1}^m$ as its columns. Then, 
\begin{enumerate}
\item\label{lem:first_condition} $\text{span}\{\bw_i\}_{i=1}^m=S$. 
\item If $\GG$ is deficiency zero, then any $m-1$ reaction vectors from the set $\{\bw_i\}_{i=1}^m$ are linearly independent.
\end{enumerate}

\end{lemma}

\begin{proof}
\begin{enumerate}
\item Note that $\text{span}\{\bw_i\}_{i=1}^m\subseteq S$. For contradiction, suppose that $\text{span}\{\bw_i\}_{i=1}^m\subsetneq S$. Then there exists a vector $\bu\in S$ that is orthogonal to $\text{span}\{\bw_i\}_{i=1}^m$. Now project all the source vertices on $\bu$, i.e. consider the set $T=\{\by_j\cdot \bu\, | \, \by_j\in V_S\}$. Let  $V_{\rm proj-max}$ denote the set of  source vertices that maximizes the dot product in $T$. Since the network is strongly endotactic, there exists a reaction with source vertex from $V_{\rm proj-max}$ to a target vertex not in $V_{\rm proj-max}$. Let $\by_1 \rightarrow \by'_1$ be this reaction. This implies that $\bu\cdot (\by'_1 - \by_1) <0$. 

Note that $\bu\cdot (\by_j - \by_1) \leq 0$ for $j=1,2,...,m$. This gives us $\bu\cdot \bw_1 = \displaystyle\sum_{\by_1\to \by_i\in E}k_{1i} \bu\cdot (\by_i - \by_1)  \leq k'_{1}\bu\cdot (\by'_1 - \by_1) <0$, contradicting that $\bu$ is orthogonal to $\text{span}\{\bw_i\}_{i=1}^m$.

\item Since $(\GG,k)$ is a strongly endotactic mass-action system, it has a steady state say ${\bx_0}$~\cite{gopalkrishnan2014geometric}[Corollary 8.12]. This gives us the following equation
\begin{eqnarray}
\displaystyle\sum_{\by_j\to \by_j' \in E} k_{\by_j\to \by_j'}{\bx_0}^{\by_j}(\by_j'-\by_j) = \displaystyle\sum_{j\in [m]}{\bx_0}^{\by_j} \bw_j = 0. 
\end{eqnarray}\
Since $\delta=0$, from Lemma~\ref{lem:strongly_endotactic_prop}.\ref{lem:first_condition} we have $\delta = 0=m-1-\rm dim(S) = m-1-\rm Im(W)$. This gives $\rm{dim}(\rm{ker}(W))=1$. Note that the vector $({\bx_0}^{\by_1}, {\bx_0}^{\by_2},..., {\bx_0}^{\by_m})$ has positive coefficients. This implies that any $m-1$ reaction vectors of $W$ are linearly independent.
\end{enumerate}

\end{proof}

\begin{remark}\label{rmk:weakly_reversible}
Lemma~\ref{lem:strongly_endotactic_prop} was shown to hold for weakly reversible reaction networks in Lemma $3.9$ in~\cite{craciun2021uniqueness}.
\end{remark}

\section{Dynamical Equivalence}\label{sec:dynamical_equivalence}

Associated with every reaction network is a dynamical system. However, it is possible for two different reaction networks to generate the same dynamical system. This phenomenon is called \textbf{dynamical equivalence}. It is also referred to as \emph{macro-equivalence}~\cite{horn1972general} or \emph{confoundability}~\cite{craciun2008identifiability}. We formally define the notion of dynamical equivalence below.

\begin{definition}\label{defn:dynamical_equivalence}
Consider two mass-action systems $(\GG,\bk)$ and $(\tilde{\GG},\tilde{\bk})$. Then $(\GG,\bk)$ and $(\tilde{\GG},\tilde{\bk})$ are said to be \textbf{dynamically equivalent} if the following holds true for all $\bx\in\mathbb{R}^n_{>0}$:
\begin{eqnarray}
 \displaystyle\sum_{\by\rightarrow \by'\in E}k_{\by\rightarrow \by'}{\bx}^{\by}(\by' - \by) =  \displaystyle\sum_{\tilde{\by}\rightarrow \tilde{\by}'\in \tilde{E}}k_{\tilde{\by}\rightarrow \tilde{\by}'}{\bx}^{\tilde{\by}}(\tilde{\by}' - \tilde{\by})
\end{eqnarray}
\end{definition}

\begin{figure}[h!]
\centering
\includegraphics[scale=0.4]{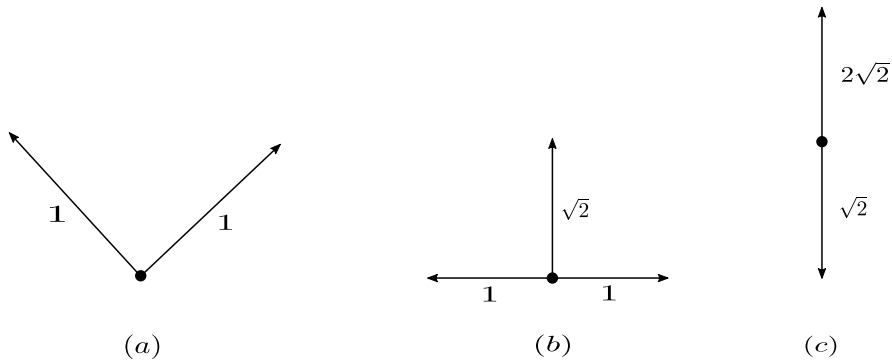}
\caption{\small The three reaction networks are dynamically equivalent, i.e., they have the same net vector given by $(0,\sqrt{2})$ corresponding to the source vertex.}
\label{fig:dynamical_equiv_networks}
\end{figure}

\begin{remark}\label{rmk:dyn_equiv_net_vectors}
It is clear from Definition~\ref{defn:dynamical_equivalence} that two mass-action systems are dynamically equivalent iff they have the same set of net reaction vectors.
\end{remark}

Figure~\ref{fig:dynamical_equiv_networks} shows a a few examples of dynamically equivalent networks.

Dynamical systems generated by reaction networks are also sometimes referred to as \textbf{realizations}. The next lemma collects a useful observation about the stoichiometric subspaces of dynamically equivalent weakly reversible reaction networks.

\begin{lemma}\label{lem:weak_stoichiometric}
Let $(\GG,\bk)$ and $(\GG',\bk')$ be two dynamically equivalent weakly reversible (strongly endotactic) realizations. Then $S(\GG) = S(\GG')$.
\end{lemma}

\begin{proof}

Since $(\GG,\bk)$ and $(\GG',\bk')$ are two dynamically equivalent weakly reversible (strongly endotactic) realizations, they have the same set of net reaction vectors (neglecting the ghost vertices). Since both these realizations are weakly reversible (strongly endotactic), by Lemma~\ref{lem:strongly_endotactic_prop} or Remark~\ref{rmk:weakly_reversible}, we get that $\text{span}\{\bw_i\}_{i=1}^m=S$ and the result follows.
\end{proof}

\begin{remark}
Note that two dynamically equivalent weakly reaction networks with the same deficiency may not have the same number of vertices or linkage classes. Figure~\ref{fig:ghost_vertex_deficiency} illustrates this point. 
\end{remark}

\begin{figure}[h!]
\centering
\includegraphics[scale=0.45]{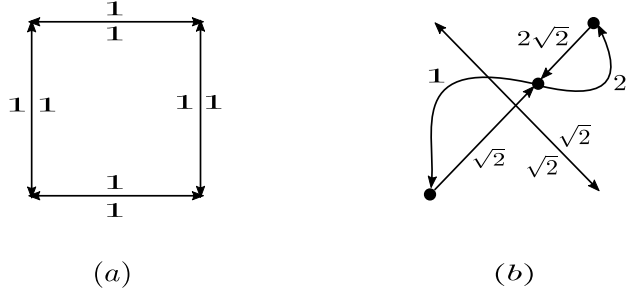}
\caption{\small (a) This network is weakly reversible, has a deficiency of one, consists of a single linkage class and contains four vertices. (b) This network is weakly reversible, has a deficiency of one, consists of two linkage classes, contains five vertices and is dynamically equivalent to the network in (a). The network in (b) possesses a ghost vertex labelled as $Q$.  }
\label{fig:ghost_vertex_deficiency}
\end{figure}

\section{Source-only realizations}\label{sec:source_only}

Network identification is a challenging problem in general due to the phenomenon of dynamical equivalence. However, certain properties of the network like complex balancing, weak reversibility are often required to design efficient systems~\cite{liptak2016kinetic,szederkenyi2013optimization}. There exists a large body of work that design algorithms for detecting properties like weak reversibility, complex balancing or for finding dense realization, i.e., realizations that have the maximum number of reactions. Many of these approaches use techniques like linear and mixed linear integer programming to perform a constrained optimization~\cite{szederkenyi2013optimization,rudan2014efficiently,rudan2014polynomial,rudan2013computing}. In such approaches, the set of source vertices is part of the input search space over which the optimization is performed. Previous works~\cite{szederkenyi2011finding,szederkenyiweak2011finding,craciun2020efficient} have shown that to deduce whether a detailed balanced, complex balanced or weakly reversible realizations exists, it suffices to check over the space of input source vertices. In our paper, we show that this property extends to \emph{endotactic, strongly endotactic, consistent and conservative networks}. Theorem~\ref{thm:main_source_only} illustrates this point.

Before we move to Theorem~\ref{thm:main_source_only}, we recall the following lemma from linear algebra which will be used in the context of \emph{consistent} networks.

\begin{lemma}\label{lem:stiemke}(Steimke's theorem)
Let $\bv_1,\bv_2,...,\bv_k \in\mathbb{R}^n$. Then exactly one the following is true
\begin{enumerate}
\item There exists constants $\lambda_1,\lambda_2,..., \lambda_k \in\mathbb{R}_{>0}$ such that $\displaystyle\sum_{i=1}^k \lambda_i \bv_i = 0$.

\item There exists a vector $\bw\in\mathbb{R}^n$ such that $\bw\cdot \bv_i\leq 0$ for all $i\in\{1,2,...,k\}$ and $\bw\cdot \bv_j < 0$ for some $j\in\{1,2,...,k\}$.
\end{enumerate}
\end{lemma}

\begin{theorem}\label{thm:consistent}
Let $\GG=(V,E)$ be a reaction network such that $\rm{Cone}\{(\by'-\by)\, |\, \by\to \by'\in E\} = \rm{Span}\{(\by'-\by) \, |\, \by\to \by'\in E \}$. Then $\GG$ is consistent.
\end{theorem}

\begin{proof}
For contradiction, assume that $\GG$ is not consistent. By Lemma~\ref{lem:stiemke}, there exists a vector $\bw\in\mathbb{R}^n$ such that $\bw\cdot (\by'-\by) \leq 0$ for all $\by\to \by'\in E$ and  $\bw\cdot (\tilde{\by}'-\tilde{\by}) < 0$ for some $\tilde{\by}\to \tilde{\by}'\in E$. Note that $\bw\cdot \bv \leq 0$ for any $\bv\in \rm{Cone}\{(\by'-\by)\, |\, \by\to \by'\in E\}$. Since $\tilde{\by}- \tilde{\by}' \in \rm{Span}\{(\by'-\by) \, |\, \by\to \by'\in E \} = \rm{Cone}\{(\by'-\by)\, |\, \by\to \by'\in E\}$, we have $\bw \cdot (\tilde{\by}- \tilde{\by}') \leq 0$ or equivalently $\bw \cdot (\tilde{\by}'- \tilde{\by}) \geq 0$, a contradiction. 

\end{proof}

\begin{remark}\label{rmk:procedure}
Using~\cite{craciun2021uniqueness} 4.8 and 4.12], we get the following: Consider a dynamical system $\dot{\bx} = f(\bx)$ generated by a weakly reversible deficiency zero network. Then the exponents of the monomials in $f(\bx)$ are the same as the vertices of the network. In particular,~\cite{craciun2021uniqueness}[Theorems 4.8 and 4.12] gives a procedure to decrease both deficiency and number of vertices of the reaction network by one, whilst maintaining dynamical equivalence. 
\end{remark}

\begin{theorem}\label{thm:main_source_only}
Choose a property $P$ from among the following:
\renewcommand{\theenumi}{(\roman{enumi})}%
\begin{enumerate}

\item Strongly endotactic.

\item Endotactic.

\item Consistent.

\item Conservative.

\end{enumerate}

Consider a mass-action system $\GG_k$. Then,  $\GG_k$ is dynamically equivalent to a mass-action system having property $P$ if and only if it is dynamically equivalent to a  mass-action system $\tilde{\GG}_{\tilde{k}}$ having property $P$ that uses only the source vertices of $\GG_{\tilde{k}}$.

\end{theorem}

\begin{proof}

\renewcommand{\theenumi}{(\roman{enumi})}%
\begin{enumerate}

\item We will first prove when the property $P$ is strongly endotactic. We will prove the contrapositive, i.e., if we have a mass-action system $\GG_k$ that is not strongly endotactic, then adding a ghost source vertex $\tilde{\by}$ to this network still implies that the resultant network is not strongly endotactic. If $\tilde{\by}$ lies inside the convex hull of the source vertices of $\GG_k$, then the resultant network is not strongly endotactic since the network $\GG_k$ was not strongly endotactic to begin with.  Note that if $\tilde{\by}$ lies outside the convex hull of the source vertices of $\GG_k$, then $\tilde{\by}$ lies on the boundary of the convex hull of the source vertices of the resultant network. Since $\tilde{y}$ is a ghost vertex, the resultant network is not strongly endotactic.

\item The proof when property $P$ is endotactic is exactly same as the proof when  property $P$ is strongly endotactic.

\item We will prove when the property $P$ is consistent. Keep on adding reaction vectors (oppositely oriented two at a time) to $\GG$ so that $\rm{Cone}\{(\by'-\by)\, |\, \by\to \by'\in E\} = \rm{Span}\{(\by'-\by) \, |\, \by\to \by'\in E \}$. By Theorem~\ref{thm:consistent}, we get that $\GG$ is consistent.

\item We will prove when the property $P$ is conservative. Since $\GG$ is conservative, there exists a positive vector $\bv$ that is orthogonal to the stoichiometric subspace $S$. Removing the ghost vertex generates a network with stoichiometric subspace $S'\subseteq S$. Therefore, we still have the positive vector $v$ that is orthogonal to $S'$ , implying that it is conservative. 

\end{enumerate}

\end{proof}

\section{Uniqueness/Non-uniqueness of realizations}\label{sec:uniqueness}

As remarked before, realizations of reaction networks may not be unique due to the property of dynamical equivalence. However, under certain restrictions on the reaction network structure, one can show that the realization is unique. For example in~\cite{craciun2021uniqueness}, Craciun et. al. showed that weakly reversible deficiency one realizations are unique. Weakly reversible deficiency zero networks form an important class of reaction networks given the rich dynamics that they exhibit. Weakly reversible deficiency zero reaction networks are known to be complex balanced for all values of the rate constants~\cite{horn1972necessary}. The property of being complex balanced implies that there exists a Lyapunov function that decreases along trajectories; further there exists a unique equilibria in each stoichiometric compatibility class which is locally asymptotically stable~\cite{horn1972general}. In the language of Feinberg~\cite{feinberg2019foundations}, the dynamical system is \emph{quasi-thermodynamic}.

\begin{theorem}[Deficiency zero theorem]\cite{horn1972necessary}
A mass-action system $(\GG,k)$ is complex balanced for all values of $\bk$ iff it is weakly reversible and has a deficiency of zero.
\end{theorem}

Note that it follows from the deficiency zero theorem that for a deficiency zero network that is not weakly reversible, there exists no positive equilibrium for all choices of $\bk$. In other words, a consistent deficiency zero reaction network is weakly reversible. Since weakly reversible deficiency zero realizations are unique~\cite{craciun2021uniqueness}, this implies that consistent deficiency zero realizations are unique.  \\

A natural extension is to ask the following question: ``Are weakly reversible deficiency one realizations unique?'' We show that the answer to this question is no in general, i.e., weakly reversible deficiency one realizations are not unique.

A network with deficiency one can be classified into the following three types:
\renewcommand{\theenumi}{(\roman{enumi})}%
\begin{enumerate}

\item\label{case:1} All linkages classes of deficiency zero.

\item One linkage of deficiency one.

\item\label{case:3} One linkage of deficiency one and the remaining linkage classes of deficiency zero.

\end{enumerate}

It is interesting to see the interplay of dynamical equivalence between the above mentioned cases in the context of weakly reversible deficiency one networks; and also the relationship between the dynamics of weakly reversible deficiency zero and deficiency one networks. \\ 

Figure~\ref{fig:zero_zero} shows an example of two dynamically equivalent weakly reversible networks that have deficiency one and each of them consists of two linkage classes of deficiency zero.

\begin{figure}[h!]
\centering
\includegraphics[scale=0.35]{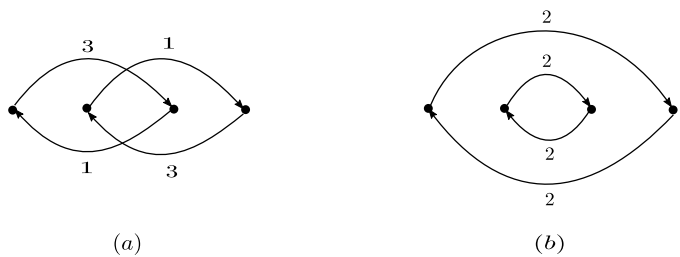}
\caption{\small This figure gives an example of two distinct weakly reversible reaction networks having a deficiency of one and consisting of two linkage classes each of deficiency zero that are dynamically equivalent. }
\label{fig:zero_zero}
\end{figure} 

Figure~\ref{fig:strongly_endotactic} shows an example of two dynamically equivalent weakly reversible networks that have deficiency one and each of them consists of a single linkage class.

\begin{figure}[h!]
\centering
\includegraphics[scale=0.32]{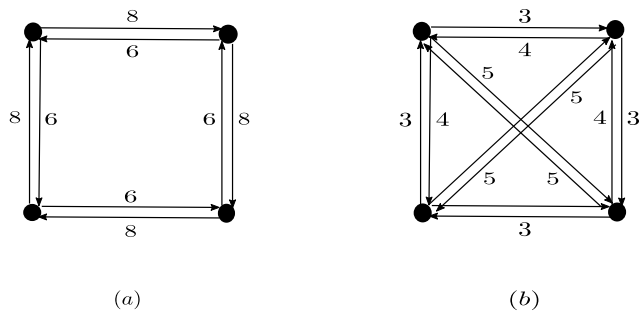}
\caption{\small This figure gives an example of two distinct weakly reversible reaction networks having a deficiency of one and consisting of a single linkage class that are dynamically equivalent.}
\label{fig:strongly_endotactic}
\end{figure} 

One can also generate examples of reaction networks consisting of different number of linkage classes. Figure~\ref{fig:linkage_one_two}.(a) shows a weakly reversible network of deficiency zero having two linkage classes. This is dynamically equivalent to the network in Figure~\ref{fig:linkage_one_two}.(b) which is a weakly reversible network of deficiency one having one linkage class.

\begin{figure}[h]
\centering
\includegraphics[scale=0.32]{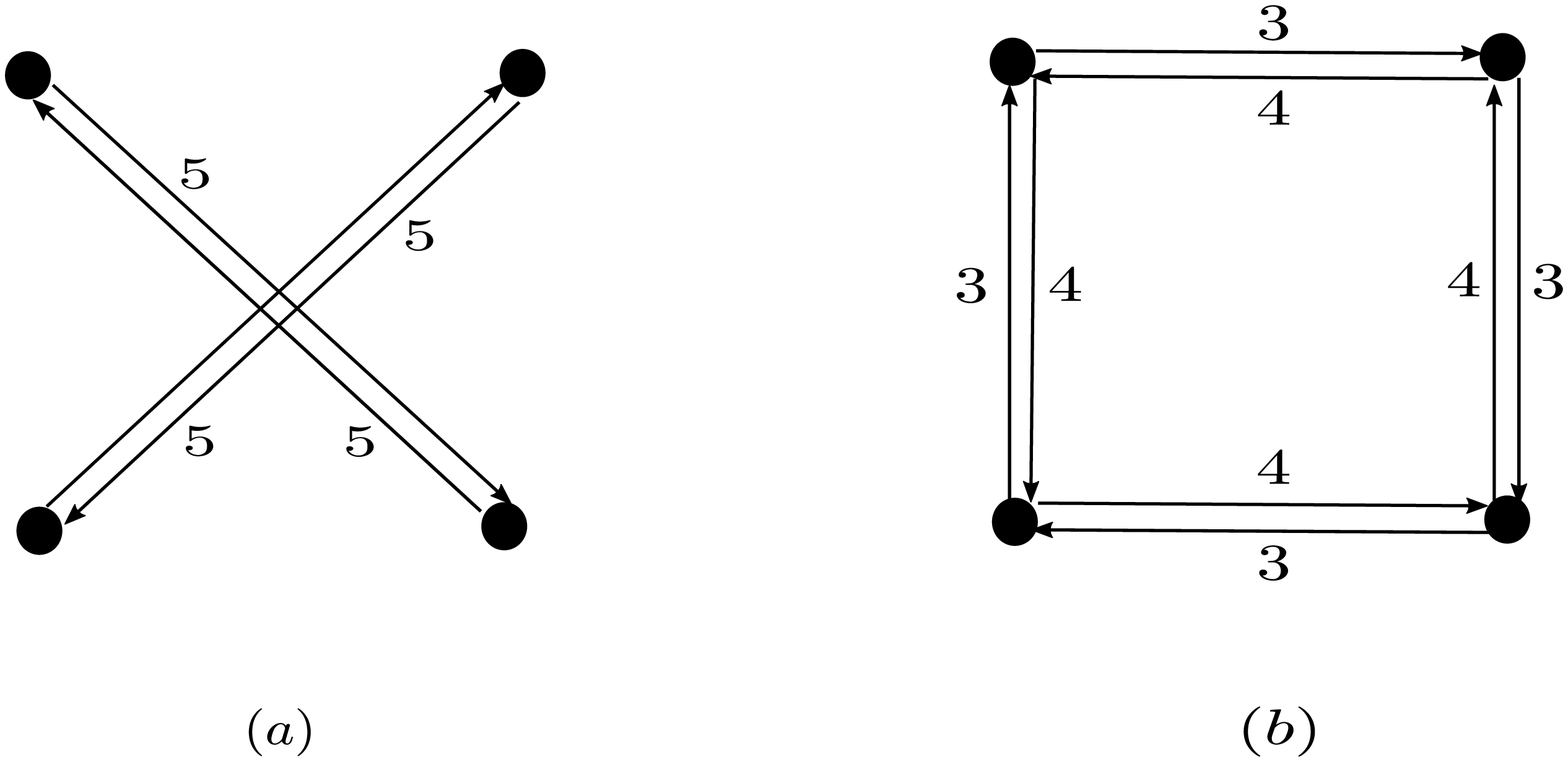}
\caption{\small The figure in (a) is a weakly reversible deficiency zero network consisting of two linkage classes that is dynamically equivalent to the network in figure (b) which is a weakly reversible deficiency one network consisting of a single linkage class.}
\label{fig:linkage_one_two}
\end{figure}

In Theorem~\ref{thm:zero_dyn_one}, we show that two reaction networks having a deficiency of one (one reaction network whose linkage classes have deficiency zero and another reaction network whose linkage classes all of deficiency zero except one which has deficiency one) and same number of linkages cannot be dynamically equivalent.  Towards this, we need a result from~\cite{craciun2021uniqueness} as the next remark states.

\begin{remark}\label{rmk:stoichiometric_independence}
The following is known from Lemma 3.10 in~\cite{craciun2021uniqueness}: If the vertices of a weakly reversible deficiency zero reaction network $\GG$ consisting of a single linkage class are split between multiple linkage classes of another reaction network $\GG'$, then the stoichiometric subspaces of the linkage classes of $\GG'$ are not linearly independent.
\end{remark}

\begin{theorem}\label{thm:zero_dyn_one}

Consider two weakly reversible mass-action systems $(\GG,k)$ and $(\GG',k')$ having a deficiency of one and the same number of linkage classes. Let the linkage classes of $(\GG,k)$ be all of deficiency zero. Let the linkage classes of $(\GG',k')$ be all of deficiency zero except one which has deficiency one. Then $(\GG,k)$ and $(\GG',k')$ cannot be dynamically equivalent.

\end{theorem}

\begin{proof}
For contradiction, assume that $(\GG,k)$ and $(\GG',k')$ are dynamically equivalent. Note that by Lemma~\ref{lem:weak_stoichiometric}, $(\GG,k)$ and $(\GG',k')$ have the same stoichiometric subspace. Since they have the same deficiency $(\delta = 1)$ and the same number of linkage classes, they have the same number of vertices. We now claim that they have the same set of vertices. Note that each linkage class $\GG$ has deficiency zero. By Corollary~\ref{cor:def_zero_ghost}, $\GG$ has no ghost vertex. Since the monomials in a dynamical system are linearly independent, $\GG$ and $\GG'$ have the same set of vertices. 

Note that the stoichiometric subspaces of the linkage classes of $\GG'$ are linearly independent. Further, since the number of linkage classes in $\GG$ and $\GG'$ are the same, there exists a deficiency zero linkage class in $\GG$ that has vertices split between multiple linkage classes of $\GG'$. By Remark~\ref{rmk:stoichiometric_independence}, this implies that the stoichiometric subspaces of the linkage classes of $\GG'$ cannot be linearly independent, a contradiction.
\end{proof}

\section{Discussion}\label{sec:discussion}

Network identification is an important problem in the analysis of biological systems. The property of dynamical equivalence makes the problem of network identification quite difficult. Identifying a particular type of realization like weakly reversible or endotactic given a dynamical system and a set of source vertices places no premature bound on the number of vertices in the network. In this paper, we show that to identify a strongly endotactic, endotactic, consistent or conservative realizations, it suffices to consider only the given source vertices. This builds upon earlier work by Craciun, Jin and Yu~\cite{craciun2020efficient}, who showed that this property is true for detailed balanced, complex balanced and weakly reversible realizations. 

In addition, we have studied the uniqueness/non-uniqueness of realizations for weakly reversible networks having low deficiency. In particular, we show that weakly reversible deficiency one realizations are not unique. This builds upon earlier work by Craciun, Jin and Yu~\cite{craciun2021uniqueness} who showed that weakly reversible deficiency zero realizations are unique. Further, we have tried to characterize the dynamical relationships between several types of weakly reversible deficiency one networks. In particular, Theorem~\ref{thm:zero_dyn_one} says that two weakly reversible deficiency one realizations such that one realization has all linkage classes of deficiency zero and the other realization has all linkage classes of deficiency zero except one which has deficiency one, cannot be dynamically equivalent.

A future line of research is to extend the set of inclusions obtained among several types of deficiency one networks in the context of dynamical equivalence to networks of higher deficiency. Another direction worthy of exploration is to design efficient algorithms to detect weakly reversible networks of high deficiency.

\section{Acknowledgements}

I thank Gheorghe Craciun for useful discussions and suggestions for improving the presentation of the paper.

\bibliographystyle{amsplain}
\bibliography{Bibliography}

\end{document}